\documentclass[]{scrartcl}
\usepackage[utf8]{inputenc}
\usepackage[letterpaper,bindingoffset=0cm,inner=2.5cm,outer=2.5cm,top=2.5cm,bottom=2.5cm]{geometry}
\usepackage{helvet}
\usepackage[T1]{fontenc}
\usepackage{graphicx}
\usepackage{amsmath,amsthm,amstext,amssymb,bm}
\usepackage{mathtools}
\usepackage{xcolor,color}
\usepackage{enumerate}
\usepackage{pgf}
\usepackage[toc,page]{appendix}
\usepackage[colorlinks=true]{hyperref}
\definecolor{mylinkcolor}{RGB}{0,0,130}
\hypersetup{colorlinks,allcolors=mylinkcolor,citecolor=mylinkcolor}
\usepackage{url}
\usepackage[capitalize,nameinlink]{cleveref}
\crefformat{equation}{(#2#1#3)}
\usepackage{booktabs}
\usepackage{tabularx}
\usepackage{enumerate}
\usepackage[shortlabels]{enumitem}
\usepackage{multirow}
\usepackage{subcaption}  
\usepackage[font=small,labelfont=bf]{caption}

\usepackage{longtable}

\usepackage{pgfplots} 
\pgfplotsset{compat=1.16}
\usetikzlibrary{positioning}

\providecommand{\keywords}[1]{\textbf{\textit{Keywords ---}} #1}

\providecommand{\acknowledgements}[1]{\textbf{\textit{Acknowledgements ---}} #1}

\newtheorem{myrem}{Remark}

\newtheorem{myprop}{Proposition}

\crefname{secname}{Section}{Section}
\Crefname{secname}{Sec.}{Sec.}
\crefname{section}{Section}{Sections}
\crefname{subsection}{Subsection}{Subsections}

\newcommand{\bbR}{\mathbb{R}} 

\newcommand{\cH}{\mathcal{H}} 
\newcommand{\cC}{\mathcal{C}} 
\newcommand{\cV}{\mathcal{V}} 

\newcommand{\bH}{\boldsymbol{H}} 
\newcommand{\bJ}{\boldsymbol{J}} 
\newcommand{\bD}{\boldsymbol{D}} 
\newcommand{\bB}{\boldsymbol{B}} 
\newcommand{\bA}{\boldsymbol{A}} 
\newcommand{\bV}{\boldsymbol{V}} 
\newcommand{\bI}{\boldsymbol{I}} 
\newcommand{\bE}{\boldsymbol{E}} 
\newcommand{\bQ}{\boldsymbol{Q}} 
\newcommand{\bX}{\boldsymbol{X}} 

\newcommand{\bx}{\boldsymbol{x}} 
\newcommand{\bu}{\boldsymbol{u}} 
\newcommand{\by}{\boldsymbol{y}}

\newcommand{\norm}[1]{ \left\| #1 \right\| }

\newcommand{\hnorm}[1]{ \norm{#1}_{\bH} }

\DeclareMathOperator{\colspan}{colspan}

\newcommand{\rN}[1]{%
  \textup{\uppercase\expandafter{\romannumeral#1}}%
}

\newcommand{\statex}{\bm{x}}                           

\newcommand{\basisprimal}{\bV}

\newcommand{\redstate}{{\statex}_{r}}           

\newcommand{\Dredstate}{\dot{{\statex}}_{r}}           

\newcommand{\approxstate}{\hat{\statex}}          

\newcommand{\error}{\bm{e}}

\newcommand{\Derror}{\dot{\error}}                

\DeclareMathOperator{\res}{\boldsymbol{\mathrm{r}}}

\newcommand{\projprim}{ \Pi_{\rN{1}}} 

\newcommand{\rederror}{{\error}_{r}}           

\newcommand{\Drederror}{\dot{{\error}}_{r}}           

\newcommand{\approxerror}{\hat{\error}}          

\newcommand{\basisalp}{\bV_{\text{A}}} 
\newcommand{\basishier}{\bV_{\text{H}}} 
\newcommand{\basisplus}{\bV_{\text{+}}} 
\newcommand{\sizealp}{n_{\text{A}}}
\newcommand{\sizehier}{n_{\text{H}}}
\newcommand{\sizeplus}{n_{\text{+}}}

\newcommand{\erroraux}{\error_{\text{A}}}                

\newcommand{\Derroraux}{\dot{\error}_{\text{A}}}                

\DeclareMathOperator{\resalp}{\boldsymbol{\mathrm{r}}_{\text{A}}}  
\DeclareMathOperator{\reshier}{\boldsymbol{\mathrm{r}}_{\text{H}}}  

\newcommand{\projalp}{ \Pi_{\text{A}}} 
\newcommand{\projhier}{ \Pi_{\text{H}}} 

\newcommand{\redstateaux}{{\statex}_{\text{H},r}} 
\newcommand{\Dredstateaux}{\dot{{\statex}}_{\text{H},r}}
\newcommand{\approxstateaux}{\hat{\statex}_{\text{H}}}

\newcommand{\errorauxhier}{\error_{\text{H}}}                
\newcommand{\Derrorauxhier}{\dot{\error}_{\text{H}}}             

\newcommand{\boundstandard}{\Delta_{\text{S}}} 
\newcommand{\boundalp}{\Delta_{\text{A}}} 
\newcommand{\boundhier}{\Delta_{\text{H}}} 

\newcommand{\effgen}{\mathtt{eff}}

\newcommand*{\tran}{^{\mkern-1.5mu\mathsf{T}}} 

\author{Johannes Rettberg\thanks{Institute of Engineering and Computational Mechanics, University of Stuttgart, Pfaffenwaldring 9, 70569 Stuttgart, Germany. (\url{johannes.rettberg,joerg.fehr@itm.uni-stuttgart.de})} \and Dominik Wittwar \thanks{Institute of Applied Analysis and Numerical Simulation, University of Stuttgart, Pfaffenwaldring 57, 70569 Stuttgart, Germany. (\url{dominik.wittwar,patrick.buchfink,robin.herkert,haasdonk@mathematik.uni-stuttgart.de})} \and Patrick Buchfink\footnotemark[2] \and Robin Herkert\footnotemark[2] \and Jörg Fehr\footnotemark[1] \and Bernard Haasdonk\footnotemark[2]}
\title{Improved a posteriori Error Bounds for Reduced port-Hamiltonian Systems}

\begin{document}
\maketitle
\begin{abstract}
\small
\textbf{Abstract} Projection-based model order reduction of dynamical systems usually introduces an error between the high-fidelity model and its counterpart of lower dimension. This unknown error can be bounded by residual-based methods, which are typically known to be highly pessimistic in the sense of largely overestimating the true error. This work applies two improved error bounding techniques, namely (a)~\textit{a hierarchical error bound} and (b)~\textit{an error bound based on an auxiliary linear problem}, to the case of port-Hamiltonian systems. The approaches rely on a second approximation of (a) the dynamical system and (b) the error system. In this paper, these methods are for the first time adapted to port-Hamiltonian systems by exploiting their structure. The mathematical relationship between the two methods is discussed both, theoretically and numerically. The effectiveness of the described methods is demonstrated using a challenging three-dimensional port-Hamiltonian model of a classical guitar with fluid-structure interaction.
\end{abstract}



%
%
\section{Introduction}
\label{Sec:introduction}
The development of modern products and the understanding of complex processes is only possible using modeling and simulation tools. In order to further optimize the products and get a detailed understanding of the processes, it is typically required to describe real examples over multiple scales and consider multiple physical domains. In an increasingly digitalized world, complex devices can considerably benefit from a digital twin throughout their life cycle, e.g.\ through increased efficiency and possible improvements through insights into data. The digital twin should be able to switch to the respective required accuracy class through a hierarchical model description. From precise descriptions for process understanding to detailed descriptions for repeated simulations to coarse models for real-time control and optimization~\cite{MehrmannUnger22}. The energy-based port-Hamiltonian (pH) framework offers ideal prerequisites for meeting the requirements of a modern simulation-based product lifecycle. The methodology allows systems to be built up modularly and to be coupled mathematically across different scales and physical domains. Through structure-preserving model reduction and flexibility in time and space discretization of the entire system or subsystems, the desire for a hierarchical model is fulfilled. Besides, under mild assumptions, these systems also satisfy helpful system-theoretical properties such as passivity and stability~\cite{MehrmannUnger22, DuindamEtAl09}.

The usual modeling process involves the spatial discretization of the continua and the partial differential equations using, for example, finite element methods that yield a high-dimensional system of ordinary differential or differential-algebraic equations. These high-fidelity models are computationally demanding, and one way of abstracting these models in a smaller subspace is projection-based model order reduction (MOR). In order to preserve the worthwhile properties of pH systems in the coarse models, this model reduction must be done in a structure-preserving manner~\cite{RettbergEtAl22}.

By approximating the high-fidelity model in a low-dimensional subspace,
MOR introduces an error that we refer to as the \emph{reduction error}.
For a pervasive simulation workflow, controlling the reduction error is essential.
Firstly, a sharply bounded error creates confidence in the simulation results.
Secondly, an error bound can be leveraged to generate adaptive methods which adaptively trade accuracy against computational resources, e.g.\ by adjusting the reduced dimension of the low-dimensional subspace~\cite{HaasdonkOhlberger09,HaasdonkEtAl22}.
This adaptivity may save time not only in the development process but also computational resources and thus yields less energy-intensive simulations.
A posteriori bounds for the error which we consider in the present work are:
\begin{itemize}
  \item the \emph{standard error bound}~\cite{HaasdonkOhlberger11},
  \item a \emph{hierarchical error bound} motivated by a suggestion in~\cite{HainEtAl19} and 
  \item the \emph{auxiliary linear problem (ALP) based error bound}~\cite{SchmidtEtAl20}.
\end{itemize}
The standard error bound uses analysis of the residual and stability constants. The hierarchical error bound uses an additional MOR solution for the state of a finer reduced model to bound the error.
The ALP error bound uses an additional reduced model for the error to derive an error bound.
The goal of the error estimation is to approximate the true error as closely as possible in order to derive a sharp and meaningful bound.
A measure of how sharp the reduction error is bounded is the \emph{effectivity} of the error bound.
The hierarchical and ALP error bounds differ from the standard error bound in that the size of the additional reduced model can be used to steer the effectivity of the error bound by increasing the computational complexity of the additional reduced model used to derive the respective error bound.

This publication shows how these error bounds can be beneficially used for high-dimensional pH systems. Our three main contributions are:
\begin{enumerate}
  \item We adapt and apply the existing error bounds to pH systems by exploiting certain properties of the pH system matrices, thereby obtaining a computable and rigorous bound,
  \item we prove that the hierarchical error bound and the ALP error bound are equivalent for linear problems with a certain choice of reduced basis and assumption on the initial conditions,
  \item we numerically compare the error bounds for a challenging three-dimensional pH model of a classical guitar with fluid-structure interaction~\cite{RettbergEtAl22}.
\end{enumerate}

The paper is organized as follows:
In \cref{sec:mor}, we briefly introduce the essentials of MOR of pH systems to fix the notation.
Subsequently, we discuss error bounds of the reduction error in MOR of pH systems in \cref{sec:error_estimation}.
Numerical experiments compare the error bounds for a challenging three-dimensional pH model of a classical guitar with approximately 5000 structural and 6300 fluid degrees of freedom in \cref{sec:numerics}. Finally, we conclude the paper in \cref{sec:conclusion_outlook}.

%
%
\section{Model order reduction of port-Hamiltonian systems}
\label{sec:mor}

A linear time-invariant (LTI) pH system in descriptor formulation~\cite{RettbergEtAl22} is considered
\begin{equation}
  \begin{aligned}
  \bE \dot{\tilde{\statex}}(t) & = ( \bJ - \bD) \bQ \tilde{\statex}(t) + \bB \bu(t), & \quad \tilde{\statex}(t_0) = \tilde{\statex}_0, \\
  \tilde{\by}(t) & = \bB\tran \bQ\tilde{\statex}(t) \label{eq:descriptor_portHamiltonian} 
  \end{aligned}
\end{equation}
with the energy-related matrices $\bE,\bQ\in\bbR^{N\times N}$ that satisfy the symmetry condition $\bE\tran\bQ = \bQ\tran\bE$ and the pH descriptor state $\tilde{\statex} \colon [t_0,T]\to \bbR^{N}$ for an initial time $t_0$ and end time $T$. If $\bE$ is non-singular, the system can be reformulated via the coordinate transformation $\statex = \bE\tilde{\statex}$ to a standard pH system
\begin{equation}
  \begin{aligned}
    \dot{\statex}(t) & = ( \bJ - \bD) \bH \statex(t) + \bB \bu(t), & \quad \statex(t_0) = \bE^{-1}\tilde{\statex}_0 = \statex_0,  \\
    \by(t) & = \bB\tran \bH\statex(t)    \label{eq:portHamiltonian}
  \end{aligned}
\end{equation}
where $\bH = \bQ\bE^{-1}$ defines the energy matrix with $0 \prec \bH = \bH\tran \in \bbR^{N \times N}$, which holds if the matrices $\bE$ and $\bQ$ commute. The Hamiltonian $\cH(\statex) := \frac{1}{2}\statex\tran\bH\statex$ specifies an energy function of the system. The matrices $\bJ = -\bJ\tran \in \bbR^{N \times N}$, $0 \preceq \bD = \bD\tran \in \bbR^{N \times N}$ and $\bB \in \bbR^{N \times m}$ describe the energy routing, dissipation, and port matrix, respectively. Furthermore, the system consists of the pH state $\statex \colon [t_0,T]\to \bbR^{N}$, the input $\bu \colon [t_0,T]\to \bbR^{m}$ and the initial state $\statex_0\in \bbR^{N}$. Additionally, we equip the space $\bbR^{N}$ with the energy inner product $\langle \cdot,\cdot\rangle_{\bH}$ and its induced energy norm $\hnorm{\cdot}$, and the space $\bbR^{N\times N}$ with the corresponding induced energy operator norm, which will also be denoted as $\hnorm{\cdot}$ given by
\begin{equation}
   \begin{aligned}
    \hnorm{\statex}^2 & := \langle \statex , \statex \rangle_{\bH} := \statex\tran \bH \statex  = \norm{ \bH^{1/2}\statex}_2 \\
    \hnorm{\bA} & := \norm{\bH^{1/2} \bA \bH^{-1/2}}_2 = \lambda_{\max}\left( \bH^{1/2}\bA \bH^{-1/2}\right), \label{eq:energy_inner_product}
   \end{aligned}  
\end{equation}
where $\bH^{1/2}$ denotes a positive definite square root of $\bH$. 

The unique solution of the first order initial value problem (IVP)~\eqref{eq:portHamiltonian} is given in closed form by
\begin{align}
    \bx(t) = \exp((\bJ-\bD)\bH (t-t_0))\bx_0 + \int\limits_{t_0}^t \exp( (\bJ-\bD)\bH(t-s))\bB \bu(s) \,\mathrm{d}s \label{eq:uniqueSolution_firstOrderSystem}.
\end{align}
Port-Hamiltonian systems implicitly exhibit many useful properties, some of which are briefly described. A pH system is a generalization of a classical Hamiltonian system where the conservation of energy is replaced by the dissipation inequality
\begin{equation}  
  \cH(\statex(t_1))-\cH(\statex(t_0))\leq\int_{t_0}^{t_1}\by(t)\tran\bu(t)\,dt\quad\text{with}\quad t_1>t_0. \label{eq:dissipationInequality}
\end{equation}
Together with the reasonable assumption that the Hamiltonian $\cH(\statex)$ is strictly positive, i.e.\ $\cH(\statex)>0$, it follows that the system is both passive and stable~\cite{SchaftJeltsema14}. Furthermore, pH systems are perfectly suitable for network-based modeling due to their modular composition. They incorporate a Dirac structure that describes a built-in power continuity. The Dirac structure guarantees that coupling two or more pH systems results again in a pH system~\cite{SchaftJeltsema14}. This is very useful when multiple subsystems are connected and holds even for systems with multiple physical domains. Many further research results exist that exploit the pH structure from different fields, e.g.\ control theory, error analysis, and optimization~\cite{RashadEtAl20}. For these reasons, it is advisable to use MOR that preserves the pH structure throughout the reduction process resulting in a pH system of much smaller dimension~\cite{Liljegren-Sailer20}. 

High-dimensional pH IVPs~\eqref{eq:portHamiltonian} often arise from a spatial semi-discretization of partial differential equations (PDE), e.g.\ with finite element methods (FEM). Those high-fidelity or full-order models (FOM) are usually computationally demanding and hence, unsuitable for e.g.\ multi-query simulations for optimization, real-time requirements of control tasks, or even too large to be computed on usual computers due to memory restrictions. For these reasons, a modern model with different hierarchical levels also includes a reduced model that avoids the abovementioned problems. One popular way of reducing the FOM is the reduction via projection, where the solution $\statex(t)$ is approximated in a subspace $\cV$ of dimension $n\ll N$ which is described by a basis matrix $\basisprimal\in\bbR^{N\times n}$ with $\colspan(\basisprimal)=\cV$. This leads to the approximation
\begin{align*}
 \statex \approx\approxstate := \basisprimal \redstate \in \bbR^N
\end{align*}
with the approximated solution $\approxstate\in\bbR^N$ and the reduced state $\redstate\in\bbR^n$. 

The structure of the pH system can be preserved through a specific Petrov-Galerkin projection as it has been investigated in~\cite{RettbergEtAl22,WolfEtAl10}. In the current work, we use the reduced pH system
\begin{equation}
  \begin{aligned}
    \basisprimal\tran\bH\basisprimal \Dredstate(t)  &= \basisprimal\tran\bH( \bJ - \bD) \bH\basisprimal \redstate(t) + \basisprimal\tran\bH\bB \bu(t) \in\bbR^n, \\ \quad \redstate(t_0) &= \basisprimal\tran\bH\statex_0 \label{eq: Reduced Port-Hamiltonian}
   \end{aligned}
\end{equation}
as it is obtained by left-multiplying the approximated system with $\basisprimal\tran\bH$ denoted as \textit{pH-preserving} reduction in~\cite{RettbergEtAl22}. The reduced matrices $\bE_r = \basisprimal\tran\bH\basisprimal$, $\bQ_r = \bI_n$, $\bJ_r = \basisprimal\tran\bH\bJ\bH\basisprimal$, $\bD_r = \basisprimal\tran\bH\bD\bH\basisprimal$ still satisfy the properties $\bJ_r = -\bJ_r\tran$, $0\prec\bD_r =\bD_r\tran$ and $\bE_r\tran\bQ_r = \bQ_r\tran\bE_r$.

%
%
\section{Error estimation of port-Hamiltonian systems}
\label{sec:error_estimation}
To the extent that not all of the dynamics take place in the low-dimensional subspace, a deviation of the reduced order model (ROM) compared to the FOM in the form of the error
\begin{align}
  \error(t) := \statex(t) - \approxstate(t)  \label{eq:error}
\end{align}
arises. The crucial task of model reduction is to find a basis $\basisprimal$ that keeps this reduction error as small as possible. However, this error is usually unknown, since the FOM state $\statex$ can not be calculated for reasons already mentioned, e.g.\ computational efficiency. Nevertheless, it is highly important to get information about the error to make statements about the quality of the simulation results of the reduced system or to use this information for adaptive basis generation schemes. For this reason, methods have been developed which provide a rigorous bound on errors based only on data in a reduced dimension. The first approaches considering a posteriori error estimation in Reduced Basis (RB) methods have been proposed for linear stationary systems~\cite{NguyenEtAl05}, which have then been extended to nonlinear~\cite{VeroyPatera05} and time-dependent problems~\cite{GreplPatera05,KnezevicEtAl11,GreplEtAl07}. This technique has been transferred to MOR of dynamical systems~\cite{HaasdonkOhlberger09} and was improved for mechanical systems~\cite{GrunertFehrHaasdonk20}. For the parametric wave equation, highly effective RB error bounds have been proposed~\cite{GlasEtAl20}. For the dissipative wave-equation we also refer to~\cite{StahlEtAl22,EggerEtAl18}. The idea of auxiliary systems for error assessment also has been applied in the iterated error system approach~\cite{AntoulasEtAl18,FengEtAl23}. In the following, we will first outline the procedure for calculating the standard error bound, which, however, usually yields results with high overestimations and is therefore only of limited use in terms of its informative value for the quality of the reduced dynamical system. For this reason, two improved methods, namely (a) a \textit{hierarchical error bound} and an error bound based on (b) an \textit{auxiliary linear problem (ALP)}, can reduce the overestimation of the error bounds. Both approaches rely on the calculation of a second approximation of (a) the dynamical system and (b) the error system. Hence, improving the error bound comes at the cost of additional computational time but is still less effort than calculating the FOM. The error bounds are adapted to the case of a pH system to exploit the pH structure and thereby further enhance the bounds.

\subsection{Standard error bound}
The error, i.e.\ the difference between full state and approximated state \eqref{eq:error}, fulfills the following IVP
\begin{align}
    \Derror(t) & = (\bJ - \bD)\bH \error(t) + \res(t)\in\bbR^N, & \quad \error(t_0) = \bm 0, \label{eq:error_system} 
\end{align}
where the initial error $\error(t_0) = \bm 0$ since we assume that $\statex_0 \in \colspan(\basisprimal)$. Generalizations exist, which allow more general $\statex_0$ or the choice of $\basisprimal$, which comes at the price of an additional term in the error bound~\cite{HaasdonkOhlberger11}. The residual of the primal system\footnote{The \textit{primal} system describes all terms that belong to the first approximation of the high-fidelity model while later the secondary approximations are introduced.}, i.e.\ the difference between the left-hand side (LHS) and right-hand side (RHS) of the approximated primal system, is given as the residual equation 
\begin{equation}
  \begin{aligned}
    \res(t) & = (\bJ-\bD)\bH\basisprimal\redstate(t) + \bB \bu - \basisprimal\Dredstate(t) \\ &= \projprim( \bJ - \bD)\bH \approxstate(t) + \projprim \bB \bu(t) \in\bbR^N, \label{eq:residual_equation_primal}
  \end{aligned}
\end{equation}
where the projection operator is defined as
\begin{align}
    \projprim & := \bI - \basisprimal \left(\basisprimal\tran\bH\basisprimal\right)^{-1}\basisprimal\tran\bH \in\bbR^{N\times N}. \label{eq:projection_primal}
\end{align}
Note, that $\projprim$ is the orthogonal projection onto the subspace $\cV^\perp$ with respect to the energy inner product~\eqref{eq:energy_inner_product}, i.e.\
\begin{align*}
  \left( \projprim \statex \right)\tran \bH\basisprimal = \bm 0 \in \bbR^{1\times n} \quad \forall\statex\in\bbR^{N}
\end{align*}
holds. The projection $\projprim$ is only well-defined if the matrix $\basisprimal\tran\bH\basisprimal\in\bbR^{n\times n}$ is invertible, which is always satisfied for a basis with $n<N$.

Analogously to \eqref{eq:uniqueSolution_firstOrderSystem}, the unique solution of \eqref{eq:error_system} is expressed by 
\begin{align}
   \error(t) = \int\limits_{t_0}^t \exp( (\bJ-\bD)\bH(t-s)) \res(s) \,\mathrm{d}s, \label{eq:unique_solution_error_system}
\end{align}
which gives the true error of the primal system. But since \eqref{eq:unique_solution_error_system} is in the high-dimensional space, the calculation of the matrix exponential is as computationally expensive as solving the FOM and hence, not feasible.

The error can be bounded by splitting the matrix-vector product in the form
\begin{align}
   \hnorm{\error(t)} \leq \max\limits_{s \in [t_0,T]} \left(\hnorm{\exp( (\bJ-\bD)\bH s)} \right) \int\limits_{t_0}^t  \hnorm{\res(s)} \,\mathrm{d}s. \label{eq:error_bound_split}
\end{align}
To circumvent the computationally demanding calculation of the norm of the matrix exponential, one can make use of the logarithmic norm~\cite{Soederlind06}
\begin{align*}
  \nu_{\ast}(\bA) = \lim_{h\to 0^+}\frac{\norm{\bI + h\bA}_{\ast}-1}{h},
\end{align*}
where $\bA$ is a square matrix, $\norm{\cdot}_{\ast}$ is an induced matrix norm, and $h\in\bbR_{>0}$. The norm of a matrix exponential can be bounded with
\begin{align}
  \norm{\exp(\bA s)}_{\ast} \leq \exp(\nu_{\ast}(\bA s)) \quad \forall s\geq 0. 
  \label{eq:logarithmic_norm_matrix_exponential}
\end{align}
In the specific case of the 2-norm, the logarithmic norm can be expressed as
\begin{align*}
  \nu_2(\bA) = \lambda_{\max}\left(\frac{\bA + \bA\tran}{2}\right),
\end{align*}
where $\lambda_{\max}$ describes the largest eigenvalue~\cite{DesoerVidyasagar09}. In the pH case, one needs to compute 
\begin{align*}
  \nu_{\bH}((\bJ - \bD)\bH s) = \nu_{2}(\bH^{1/2}(\bJ - \bD)\bH \bH^{-1/2}s).
\end{align*}
The pH structure is exploited by considering that $\bJ$ is skew-symmetric and therefore vanishes, leaving only the symmetric part 
\begin{align}
  \nu_{\bH}((\bJ - \bD)\bH s) = \lambda_{\max}(-\bH^{1/2}\bD\bH^{1/2}s) \leq 0, \label{eq:logarithmic_norm_portHamiltonian}
\end{align}
which can be bounded by zero due to $-\bD\preceq0$. Note that, if $-\bD\prec 0$ one could precalculate the largest eigenvalue to obtain an even better constant. Inserting \eqref{eq:logarithmic_norm_portHamiltonian} into \eqref{eq:logarithmic_norm_matrix_exponential} leads to 
\begin{align}
   \hnorm{\exp( (\bJ-\bD)\bH s)} \leq \exp(\nu_{\bH}((\bJ - \bD)\bH s))\leq 1. \label{eq:LogarithmicNorm} 
\end{align}

The standard error bound is therefore determined by the integral of the residual
\begin{align}
  \hnorm{\error(t)} \leq \boundstandard(t) :=  \int\limits_{t_0}^t  \hnorm{\res(s)} \,\mathrm{d}s. \label{eq:standard_error_bound}
\end{align}
Obviously, this error bound can only be monotonically increasing over time, even if the true error should occasionally decay. This makes the standard error bound typically pessimistic. 
\subsection{Auxiliary linear problem (ALP) error bound}
Since the standard bound often highly overestimates the error, there is a great desire for improvement. One approach to generate a posteriori error estimates of arbitrary good effectivity for general nonlinear, steady and unsteady problems, has been introduced as \textit{auxiliary linear problem (ALP)} based error bounds. The method is based on an approximation of the general nonlinear error system by a linearization and subsequent MOR~\cite{SchmidtEtAl20}. This reduction uses a second projection basis, the ALP basis $\basisalp\in\bbR^{N\times \sizealp}$ with $\sizealp\ll N$. In our present case of a linear FOM, the ALP system is directly the linear error system~\eqref{eq:error_system}, which is approximated in a second subspace $\colspan(\basisalp) = \cV_A$. The approximated error is described by
\begin{align*}
    \approxerror(t) = \basisalp \rederror(t) \approx \error(t) \in \bbR^N,
\end{align*}
with the reduced coordinates $\rederror(t)\in\bbR^{\sizealp}$. Note also that the validity criterion of~\cite{SchmidtEtAl20}, which is required in the nonlinear case, is always satisfied for a linear FOM, hence can be ignored in the following. In the linear case the error system~\eqref{eq:error_system} can be interpreted as a pH system with the error as state variable and the residual as the input. Hence, one can once more use structure-preserving MOR that leads to a reduced error system
\begin{align}
    \basisalp\tran\bH\basisalp \Drederror(t) & = \basisalp\tran \bH( \bJ - \bD) \bH\basisalp \rederror(t) + \basisalp\tran\bH\res(t) \in\bbR^{\sizealp}, & \quad \rederror(t_0) = \bm 0, \label{eq: ALP Port-Hamiltonian} 
\end{align}
which still is a pH system. The reduced initial error $\rederror(t_0)$ vanishes due to $\error(t_0)=\bm 0$. One can now define a second error $\erroraux(t) \in \bbR^N$ between the true error and the approximated error as
\begin{align}
    \erroraux(t) := \error(t) - \approxerror(t) \in\bbR^N. \label{eq:alp_general_error} 
\end{align}
Analogously to \eqref{eq:error_system}, the pH IVP for the second error reads
\begin{align}
    \Derroraux(t) & = (\bJ - \bD)\bH \erroraux(t) + \resalp(t) \in \bbR^N, & \quad \erroraux(t_0) = \bm 0, \label{eq:error_alp_portHamiltonian} 
\end{align}
where the second residual $\resalp$ is exclusively defined by values of the error approximation as
\begin{align*}
    \resalp(t) & = \projalp( \bJ - \bD)\bH \approxerror(t) + \projalp \res(t) \in\bbR^N, 
\end{align*}
and the projection operator $\projalp\in\bbR^{N\times N}$ is defined by
\begin{align*}
\projalp & = \bI - \basisalp \left(\basisalp\tran\bH\basisalp\right)^{-1}\basisalp\tran\bH \in\bbR^{N\times N}, 
\end{align*}
where $\basisalp\tran\bH\basisalp$ needs to be invertible which is satisfied if $\operatorname{colrank}(\basisalp)=\sizealp$.
The auxiliary error system \eqref{eq:error_alp_portHamiltonian} can be uniquely solved with
\begin{align}
\erroraux(t) = \int\limits_{t_0}^t \exp( (\bJ-\bD)\bH(t-s)) \resalp(s) \,\mathrm{d}s \label{eq:error_alp_unique_solution}
\end{align}
and hence the same bound as in \eqref{eq:error_bound_split} can be applied by making use of the logarithmic norm \eqref{eq:LogarithmicNorm}
   \begin{align*}
   \hnorm{\erroraux(t)} \leq \underbrace{\max\limits_{s \in [t_0,T]} \hnorm{\exp( (\bJ-\bD)\bH s)}}_{\leq 1 \, (\text{cf.}\, \eqref{eq:LogarithmicNorm})} \int\limits_{t_0}^t  \hnorm{\resalp(s)} \,\mathrm{d}s.
   \end{align*}
From \eqref{eq:alp_general_error} we obtain the relation
   \begin{align*}
   \error(t) = \approxerror(t) + \erroraux(t) \in\bbR^N.
   \end{align*}
Using the energy norm \eqref{eq:energy_inner_product} and the triangle inequality, this converts to
   \begin{align*}
   \hnorm{\error(t)} = \hnorm{\approxerror(t) + \erroraux(t)} \leq \hnorm{\approxerror(t)} + \hnorm{\erroraux(t)}.
   \end{align*}
Finally, we obtain a computable error bound for the primal error $\error(t)$ based on the auxiliary linear problem \eqref{eq:error_system}
  \begin{align}
    \hnorm{\error(t)} \leq \boundalp(t) := {\hnorm{\approxerror(t)}} + {\int\limits_{t_0}^t  \hnorm{\resalp(s)} \,\mathrm{d}s }. \label{eq:alp_bound}
  \end{align}
Provided that the second approximation $\erroraux$ is of sufficient quality, the calculation should result in a highly effective error bound~\cite{SchmidtEtAl20}.

\subsection{Hierarchical error bound}
Another approach to counteracting the high overestimation of the standard error bound, motivated by a suggestion in~\cite{HainEtAl19}, is again using a second approximation. Only this time, the original pH system \eqref{eq:portHamiltonian} is approximated rather than the error system. The system is approximated in a second subspace $\colspan(\basishier)$ with the hierarchical basis matrix $\basishier\in\bbR^{N \times \sizehier}$. The approximated solution is defined as
\begin{align*}
    \approxstateaux(t) = \basishier \redstateaux(t) \approx \statex(t) \in\bbR^N,
\end{align*}
where $\redstateaux(t)\in\bbR^{\sizehier}$ are the reduced coordinates. The second system needs to be of better accuracy than the primal system which is ensured by $\sizehier > n$ and just adding additional $\sizehier-n$ basis vectors to $\basisprimal$. Using the \textit{pH-preserving} projection yields the reduced hierarchical system
\begin{equation}
  \begin{aligned}
  \basishier\tran\bH\basishier \Dredstateaux(t)  &= \basishier\tran\bH( \bJ - \bD) \bH\basishier \redstateaux(t) + \basishier\tran\bH\bB \bu(t) \in\bbR^{\sizehier}, \\ \quad \redstateaux(t_0) &= \basishier\tran\bH\statex_0. \label{eq:reduced_hierarchical_system}
  \end{aligned}
\end{equation}
The idea is in some sense analogous to error estimation in time discretization schemes, e.g.\ the Runge-Kutta scheme of 4th order with variable step-size\footnote{in Matlab known as ode45}, where the error is estimated from a 5th order scheme of higher accuracy and the step-size is optimized with respect to this estimation.

The bound is then obtained by the error definition~\eqref{eq:error}, adding a zero,
\begin{align*}
  \error(t) = \statex(t) - \approxstate(t) = \statex(t) \underbrace{- \approxstateaux(t) + \approxstateaux(t)}_{=\,\bm 0} - \approxstate(t) \in \bbR^N,
\end{align*}
and utilizing the energy norm~\eqref{eq:energy_inner_product} and the triangle inequality
\begin{align*}
    \hnorm{\error(t)} = \hnorm{\statex(t) - \approxstate(t)} \leq \hnorm{\approxstateaux(t) - \approxstate(t)} + \hnorm{\statex(t) - \approxstateaux(t)}.
\end{align*}
The first term in the hierarchical error bound is the difference between the two approximated systems. Recognizing that the second summand can be bounded by the standard error bound~\eqref{eq:standard_error_bound} of the more accurate approximate $\approxstateaux(t)$, the hierarchical error bound can be written as
\begin{align}
  \hnorm{\error(t)} \leq \boundhier(t) := \hnorm{\approxstateaux(t) - \approxstate(t)} + {\int\limits_{t_0}^t  \hnorm{\reshier(s)} \,\mathrm{d}s }, \label{eq:hier_bound}
\end{align}
where the residual is calculated as 
\begin{align*}
  \reshier(t) & = \projhier( \bJ - \bD)\bH \approxstateaux(t) + \projhier \bB \bu(t) \in\bbR^N 
\end{align*}
and the projection operator can be obtained by inserting the basis $\basishier$ into \eqref{eq:projection_primal} which leads to
\begin{align}
  \projhier & := \bI - \basishier \left(\basishier\tran\bH\basishier\right)^{-1}\basishier\tran\bH \in\bbR^{N\times N} \label{eq:projection_hier}
\end{align}
with $\basishier\tran\bH\basishier$ being invertible which is satisfied if $\operatorname{colrank}(\basishier) = \sizehier$. Note that this error bound is analogous to the additive decomposition mentioned in~\cite{HainEtAl19} for general inf-sup stable parametric PDEs. In that paper also, another structure of a hierarchical error estimator is suggested, which is based on removing the residual summand and instead multiplying the fine-to-coarse ROM error by a small factor determined by a saturation constant. The resulting estimator, however, only is a rigorous bound under a so-called saturation assumption, which is hard to verify. This is why we focus on the above additive decomposition in \eqref{eq:hier_bound}, a rigorous error bound by construction.

\subsection{Relationship between hierarchical and ALP error bound}
The derivation of the equations in the previous subsections and the very similar structure of the equation components suggests that there is a relationship between the ALP and the hierarchical bound. The following proposition analyzes under which conditions the two bounds lead to the same result.

\begin{myprop}
  \label{prop:hier_equals_ALP}
  Consider a pH system~\eqref{eq:portHamiltonian} and let $\redstate\in\cC^1([t_0,T],\bbR^n)$ be the solution of the reduced pH system~\eqref{eq: Reduced Port-Hamiltonian} with an induced error $\error\in\cC^1([t_0,T],\bbR^N)$ given by~\eqref{eq:error}, $\rederror\in\cC^1([t_0,T],\bbR^{\sizealp})$ the solution of the reduced ALP problem~\eqref{eq: ALP Port-Hamiltonian} with a second error $\erroraux\in\cC^1([t_0,T],\bbR^N)$ given by~\eqref{eq:alp_general_error} which can be solved by~\eqref{eq:error_alp_portHamiltonian} and $\redstateaux\in\cC^1([t_0,T],\bbR^{\sizehier})$ the solution of the reduced hierarchical system~\eqref{eq:reduced_hierarchical_system}.

   If $\basisalp = \basishier = \begin{bmatrix}
    \basisprimal & \basisplus
  \end{bmatrix}$, where $\basisplus\in \bbR^{N\times\sizeplus}$ with $\sizeplus = \sizealp - n = \sizehier-n$ and furthermore, for the initial conditions $\rederror(0) = \redstateaux(0) - \begin{bmatrix}
    \redstate(0) \\ \bm 0_{\sizeplus\times 1}
  \end{bmatrix}$ is satisfied, then it holds that
  \begin{align}
    \boundalp(t) = \boundhier(t) \in\bbR_{\geq0},\quad \forall t\in [t_0,T],
  \end{align}
  with the ALP bound $\boundalp(t)$ from~\eqref{eq:alp_bound} and the hierarchical bound $\boundhier(t)$ from~\eqref{eq:hier_bound}.
\end{myprop}
\begin{proof}
  We want to show that under the assumptions in \cref{prop:hier_equals_ALP}
  \begin{align*}
    \boundalp(t) &= \boundhier(t) \quad \forall t\in[t_0,T] 
  \end{align*}
  which is equivalent to
  \begin{align*}
    {\hnorm{\approxerror(t)}} + {\int\limits_{t_0}^t  \hnorm{\resalp(s)} \,\mathrm{d}s }  &= \hnorm{\approxstateaux(t) - \approxstate(t)} + {\int\limits_{t_0}^t  \hnorm{\reshier(s)} \,\mathrm{d}s }
  \end{align*}
  which is satisfied if we show that the respective summands $\approxerror(t) = \approxstateaux(t) - \approxstate(t)$ and $\resalp(t) = \reshier(t)$ are equal. First, we want to show that $\approxerror(t) = \approxstateaux(t) - \approxstate(t)$. 

  Left multiplying~\eqref{eq:residual_equation_primal} with $\basisalp\tran\bH$ yields 
  \begin{align}
    \basisalp\tran\bH\res(t) & = \basisalp\tran\bH(\bJ-\bD)\bH\basisprimal\redstate(t) + \basisalp\tran\bH\bB \bu(t) - \basisalp\tran\bH\basisprimal\Dredstate(t) \label{eq:primal_residual_alp_projected}
  \end{align}
  and from \eqref{eq:reduced_hierarchical_system}, it follows that
  \begin{align}
    \basishier\tran\bH\bB \bu(t) =  \basishier\tran\bH\basishier \Dredstateaux(t) - \basishier\tran\bH( \bJ - \bD) \bH\basishier \redstateaux(t). \label{eq:alp_projected_input}
  \end{align}
  Using $\basishier = \basisalp$ we can insert \eqref{eq:alp_projected_input} into \eqref{eq:primal_residual_alp_projected} and obtain
  \begin{equation}
    \begin{aligned}
    \basisalp\tran\bH\res(t)  = \, &\basisalp\tran\bH(\bJ-\bD)\bH\basisprimal\redstate(t) - \basisalp\tran\bH\basisprimal\Dredstate(t) + \basisalp\tran\bH\basisalp \Dredstateaux(t) \\ &- \basisalp\tran\bH( \bJ - \bD) \bH\basisalp \redstateaux(t). \label{eq:alp_res_hier_diff_mixed_form}
    \end{aligned}
  \end{equation}
  Furthermore, the approximated state can be reformulated as 
  \begin{align*}
    \basisprimal\redstate(t) = \basisprimal\redstate(t) + \basisplus \bm 0_{\sizeplus\times 1} = \basisalp\begin{bmatrix}
      \redstate(t) \\ \bm 0_{\sizeplus \times 1}
    \end{bmatrix},
  \end{align*} 
  where $\bm 0_{\sizeplus \times 1}$ is a zero column vector of size $\sizeplus$ and therefore, \eqref{eq:alp_res_hier_diff_mixed_form} can be rewritten as
  \begin{equation*}
    \begin{aligned}
    \basisalp\tran\bH\basisalp \left(\Dredstateaux(t) - \begin{bmatrix}
      \Dredstate(t) \\ \bm 0_{\sizeplus\times 1}
    \end{bmatrix}\right) = &\basisalp\tran\bH(\bJ - \bD)\bH\basisalp \left(\redstateaux(t) - \begin{bmatrix}
      \redstate(t) \\ \bm 0_{\sizeplus\times 1}
    \end{bmatrix}\right) \\ &+ \basisalp\tran \bH\res(t).
  \end{aligned}
\end{equation*}
  Due to the assumption that
  \begin{align*}
    \rederror(0) = \redstateaux(0) - \begin{bmatrix}
      \redstate(0) \\ \bm 0_{\sizeplus\times 1}
    \end{bmatrix},
  \end{align*}
  $\left(\redstateaux(t) - \begin{bmatrix}
    \redstate(t) \\ \bm 0_{\sizeplus\times 1}
  \end{bmatrix}\right)$ solves the IVP~\eqref{eq: ALP Port-Hamiltonian} and it holds that for all $t$
  \begin{align*}
    \rederror(t) = \redstateaux(t) - \begin{bmatrix}
      \redstate(t) \\ \bm 0_{\sizeplus\times 1}
    \end{bmatrix}.
  \end{align*}
  Hence, it follows
  \begin{align}
    \basisalp\rederror(t) &= \basisalp \redstateaux(t) - \basisalp \begin{bmatrix}
      \redstate(t) \\ \bm 0_{\sizeplus\times 1}
    \end{bmatrix} = \basishier\redstateaux(t) - \begin{bmatrix}
      \basisprimal & \basisplus
    \end{bmatrix} \begin{bmatrix}
      \redstate(t) \\ \bm 0_{\sizeplus\times 1}
    \end{bmatrix} \nonumber \\ &= \basishier\redstateaux(t) - \basisprimal\redstate(t) \label{eq:approxalp_equal_approxhier_with_basis}
  \end{align}
  that can be equivalently written as
  \begin{align}
    \approxerror(t) = \approxstateaux(t) - \approxstate(t), \label{eq:approxalp_equal_approxhier}
  \end{align}
  which concludes the first part of the proof. It remains to show that $\resalp(t) = \reshier(t)$. From ~\eqref{eq:error} and \eqref{eq:alp_general_error} we know that
  \begin{align*}    
    \error(t) = \statex(t) - \basisprimal\redstate(t) \quad \text{and}\quad
    \erroraux(t) = \error(t) - \basisalp\rederror(t)
  \end{align*}
  and similarly, we can define a second error for the hierarchical system
  \begin{align*}
    \errorauxhier(t) = \statex(t) - \basishier\redstateaux(t)
  \end{align*}
  which can be equivalently solved with the IVP
  \begin{align}
    \Derrorauxhier(t) & = (\bJ - \bD)\bH \errorauxhier(t) + \reshier(t), & \quad \errorauxhier(t_0) = \bm 0. \label{eq:error_hier_system}
  \end{align}
  Subtracting the error expressions yields
  \begin{align}
    \error(t) - \errorauxhier(t) \overset{\eqref{eq:alp_general_error}}{=} \basisprimal\rederror(t) + \erroraux(t) - \errorauxhier(t) = \basishier\redstateaux(t) - \basisprimal\redstate(t). \label{eq:diff_error_errorhier}
  \end{align}
  If we subtract~\eqref{eq:approxalp_equal_approxhier_with_basis} from \eqref{eq:diff_error_errorhier}, it holds that
  \begin{align*}
    \erroraux(t) = \errorauxhier(t)
  \end{align*}
  and hence, also their time derivatives
  \begin{align*}
    \Derroraux(t) = \Derrorauxhier(t).
  \end{align*}
  Comparing~\eqref{eq:error_hier_system} and \eqref{eq:error_alp_portHamiltonian} shows that
  \begin{align*}
    \resalp(t) = \reshier(t)
  \end{align*}
  which concludes the proof.
\end{proof}

\begin{myrem}
  Note that in the derivation of \cref{prop:hier_equals_ALP} no specific properties of pH system have been used and hence, the results can also be used for more general linear dynamical system with non-structure-preserving MOR.
\end{myrem}

The additional condition on the initial value shall be further analyzed. Therefore, some reformulations are performed as
\begin{align*}
  \rederror(0) &= \redstateaux(0) - \begin{bmatrix}
    \redstate(0) \\ \bm 0_{\sizeplus\times 1}
  \end{bmatrix} \\
  \iff \basisalp\tran\bH(\statex_0-\basisprimal\basisprimal\tran\bH\statex_0) &= \basisalp\tran\bH\statex_0 - \begin{bmatrix}
   \basisprimal\tran\bH \redstate(0) \\ \bm 0_{\sizeplus\times 1}
  \end{bmatrix} \\
   \iff \basisalp\tran\bH\underbrace{\basisprimal\basisprimal\tran\bH\statex_0}_{=\basisalp\begin{bmatrix}
   \basisprimal\tran\bH \redstate(0) \\ \bm 0_{\sizeplus\times 1}
  \end{bmatrix}} &= \begin{bmatrix}
    \basisprimal\tran\bH \redstate(0) \\ \bm 0_{\sizeplus\times 1}
   \end{bmatrix} \\
\iff \left(\basisalp\tran\bH\basisalp-\bI \right) \begin{bmatrix}
  \basisprimal\tran\bH \redstate(0) \\ \bm 0_{\sizeplus\times 1}
 \end{bmatrix} &= \bm 0.
\end{align*}
From this, one can formulate three properties that ensure the initial value condition in each case:
\begin{enumerate}
  \item The secondary basis matrix, and secondary projection matrix are biorthogonal, which in our case is equivalent to the secondary basis being orthogonal with respect to the energy inner product, i.e.\ $\basisalp\tran\bH\basisalp=\bI$
  \item The initial condition is included in the primal basis $\statex_0\in\colspan(\basisprimal)$ which results in the same reduced vectors for the primal and hierarchical system $\redstateaux(0) = \begin{bmatrix}
    \redstate(0) \\ \bm 0_{\sizeplus\times 1}
  \end{bmatrix}$ and zero initial conditions of the reduced error $\rederror(0)=\bm 0$
  \item The system has zero initial conditions $\statex_0= \bm 0$ which is a special case of the second item since the zero vector is always included in the basis.
\end{enumerate}

%
%
\section{Results}
\label{sec:numerics}
We test the proposed improved error bounds for pH systems on a fluid-structure interaction model of a classical guitar, see \cref{fig:GuitarCutView} that is derived from the equations of linear elasticity for the guitar body and the wave equation for the enclosed air in the guitar body~\cite{RettbergEtAl22}. The system is of size $N=11248$ with $4908$ structural and $6340$ fluid degrees of freedom (DOFs). The pH matrices can be accessed at~\cite{RettbergEtAl23}. The top plate of the guitar is excited by a time-dependent force $u(t)=\hat{u}\sin(\omega t) $ at the location of the guitar bridge. The circular frequency $\omega = 2\pi f$ lies in the frequency range $f = [82,320]~\text{Hz}$ and the amplitude is $\hat{u}=1~\text{N}$. Throughout the paper, a pH system in the form of \eqref{eq:portHamiltonian} is used which is denoted as \textit{momentum formulation} in~\cite{RettbergEtAl22}.

\begin{figure}[htb]
	\centering
    \input{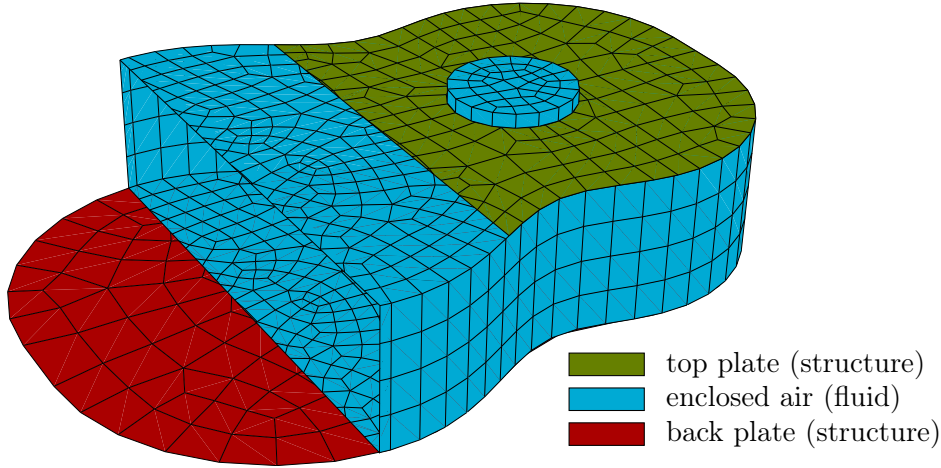}
	\caption{Sectional view of the FEM multi-physics model of a classical guitar~\cite{RettbergEtAl22}.}
	\label{fig:GuitarCutView}
\end{figure}

A reasonable quality indicator for assessing the quality of error estimators or error bounds is the effectivity which measures how close the error bound is to the true error. The effectivity is defined as 
\begin{align*}
  \effgen_k(t) := \frac{\Delta_k(t)}{\hnorm{\error(t)}},\quad k\in\{S,A,H\},
\end{align*}
where $k$ is a placeholder for the standard, ALP and hierarchical bound. Furthermore, we define the maximum (worst) effectivity as
\begin{align}
  \effgen_{\max,k} := \max\limits_{t \in [t_0,T]} \effgen_k(t),\quad k\in\{S,A,H\}. \label{eq:maximum_effectivities}
\end{align}
A rigorous error bound satisfies $\effgen_k\geq 1$ and the closer the bound is to one, the sharper the bound is. 

In the first experiment, the guitar is excited for a time interval $[t_0,T]=[0,0.01]~\text{s}$ which corresponds to approximately one period in the mentioned frequency interval. The data-based bases are generated from one trajectory for $f=100~\text{Hz}$ with $n_s = 1000$ time steps and the results are conducted for a test input of $f = 320~\text{Hz}$. For the basis generation, we use an energy-weighted Proper Orthogonal Decomposition (POD) using the full state snapshots $\bX_s:=(\statex_i)_{i=1}^{n_s}$, denoted as \textit{POD-State} in~\cite{RettbergEtAl22}. The primal basis size is set to $n=120$ while two secondary basis sizes $\sizealp=\sizehier\in\{200,400\}$ are investigated. The hierarchical basis $\basishier$ extends the primal basis $\basisprimal$ by additional POD-modes. The ALP basis $\basisalp$ is obtained from error snapshots $\bE_s:=(\error_i)_{i=1}^{n_s}$.

The results are illustrated in \cref{fig:SharpBound}, where the true error and the bounds are shown in the left axis while the effectivities are illustrated in right axis. Note that in order to better distinguish between the different bounds close to 1, the quantity $\effgen_k-1$ is shown. In general, one can observe that all bounds satisfy the rigorous bound property of $\effgen_k\geq 1$. The behavior of the standard bound discussed in the introduction is apparent, as it overestimates the true error by several orders of magnitude and is therefore not suitable for the typical uses of an error estimator. Additionally, the usual monotonicity property leads to a worsening of the effectivity over time. In contrast, both the hierarchical $\boundhier(t)$ and the ALP error bound $\boundalp(t)$ manage to improve the standard bound by several orders of magnitude, already for $n_{\text{A/H}}=200$. ALP and hierarchical bounds describe very similar characteristics for the same size of their bases. Differences can be seen in the effectivities for the case of $n_{\text{A/H}}=400$, where it seems advantageous to include information about the error dynamics into the secondary ALP basis. Even the monotonically increasing behavior can be overcome, and the bounds follow the true error very closely over the entire time interval.

\begin{figure}
  \centering
  \input{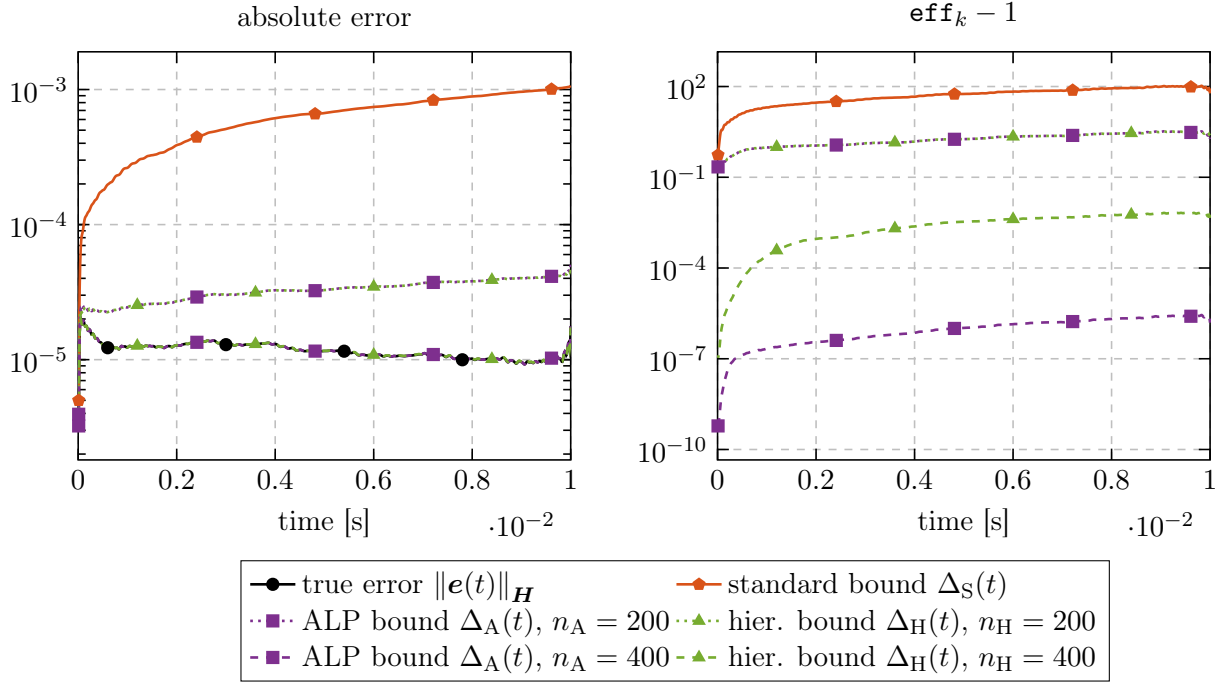}
  \caption{Error bounds and effectivities for a \textit{POD-State} basis generation with a primal basis of $n=120$, a hierarchical basis with additional POD-modes and an ALP basis from error snapshots for secondary basis sizes $n_k \in \{200,400\}$ for $k\in\{\text{A},\text{H}\}$ and a simulation time $T=0.01~\text{s}$.}
  \label{fig:SharpBound}
\end{figure}

In \cref{prop:hier_equals_ALP}, it was proven that the hierarchical and ALP error bound coincide if the conditions $\basisalp=\basishier=\begin{bmatrix}
  \basisprimal & \basisplus
\end{bmatrix}$ and $\rederror(0) = \redstateaux(0) - \begin{bmatrix}
  \redstate(0) \\ \bm 0_{\sizeplus\times 1}
\end{bmatrix}$ are satisfied. The second experiment is dedicated to show this result numerically. Therefore, three different basis generation techniques were used, namely the aforementioned \textit{POD-State}, and additionally \textit{C-SVD} and \textit{SVD-like}. The methods \textit{C-SVD}~\cite{PengMohseni16} and \textit{SVD-like}~\cite{PengMohseni16,BuchfinkEtAl19} are variants of the Proper Symplectic Decomposition (PSD) where the obtained bases satisfy symplectic properties. While \textit{C-SVD} builds upon an adapted complex snapshot matrix and leads to a symplectic, orthogonal basis~\cite{PengMohseni16}, the \textit{SVD-like} is constructed by computing a special decomposition leading to symplectic, non-orthogonal basis vectors~\cite{PengMohseni16,BuchfinkEtAl19}. Throughout all variants, we calculate a secondary basis $\basishier=\basisalp$ of size $\sizehier=\sizealp$ and take the first $n$ basis vectors for $\basisprimal$. Since zero initial conditions are used, the requirement on the initial values in \cref{prop:hier_equals_ALP} is always valid. In the experiments, different combinations of primal and secondary basis sizes were investigated. Furthermore, the simulation end time is increased to $T=0.1~\text{s}$, which implies a much more complex reproduction of the dynamics that now include several oscillations.

The results are illustrated in \cref{fig:EqualityHierALP} where the maximum effectivities~\eqref{eq:maximum_effectivities} of the ALP bound are displayed on the horizontal axis and the corresponding maximum effectivities for the hierarchical bound are on the vertical axis. In order to numerically verify \cref{prop:hier_equals_ALP}, the square marker with these coordinates needs to lie on the diagonal. All investigated combinations of basis generation techniques, primal and secondary basis sizes satisfy this property and hence, the theoretical results are validated numerically. Further information is provided in the figure, firstly stating that \textit{POD-State} and \textit{SVD-like} lead to very similar results, which is also in line with the results from~\cite{RettbergEtAl22} where these basis generation techniques showed the best reduction results and this also seems to be reflected for the effectivities of the error bounds. On the other hand, the effectivities are larger than from the first experiment, which is due to the more complex dynamics. \\

\begin{figure}[htb]
    \centering
    \input{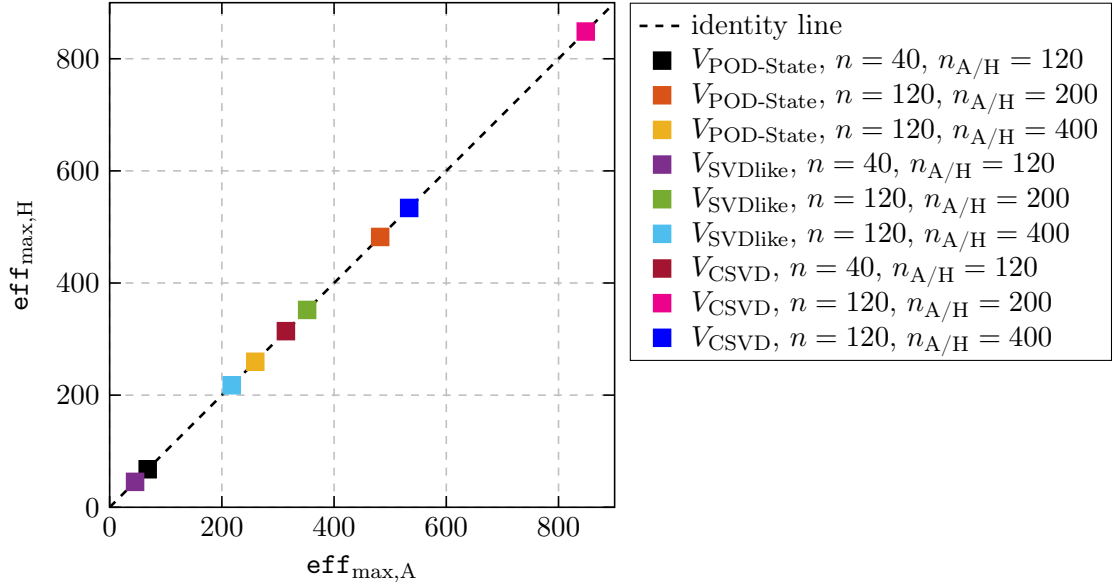}
    \caption{Comparison of hierarchical and ALP error bound with conditions from \cref{prop:hier_equals_ALP}.}
    \label{fig:EqualityHierALP}
  \end{figure}

  The third experiment analyzes the effects of more complex dynamics on the error bounds. For this, the bounds are subdivided into their components: The ALP bound consists of the approximated error $\hnorm{\approxerror}$ and the integral over the ALP residual ${\int\limits_{t_0}^t  \hnorm{\resalp(s)} \,\mathrm{d}s }$, while the hierarchical bound is composed of the difference of the approximated system states $ \hnorm{\approxstateaux(t) - \approxstate(t)} $ and the integral over the hierarchical residual ${\int\limits_{t_0}^t  \hnorm{\reshier(s)} \,\mathrm{d}s }$. In \cref{fig:BoundsSubparts}, the parameters of the first experiment are reused. Only that in the right axis of \cref{fig:BoundsSubparts}, the experiment is conducted for a longer simulation interval of $[t_0,T] = [0,0.1]~\text{s}$. During this increased time, multiple oscillation periods occur, acoustic pressure waves travel through the domain and are reflected at the boundaries, and energy is transferred between the structure and the fluid. It is not possible to fully reproduce these trajectories in a small number of basis vectors which is also reflected in an increased value of the true error compared to the experiment with the short simulation time. Even though the dynamics are hard to capture, the improved methods still outperform the standard bound by about one order of magnitude. 

  \begin{figure}[htb]
  	\centering
  	\input{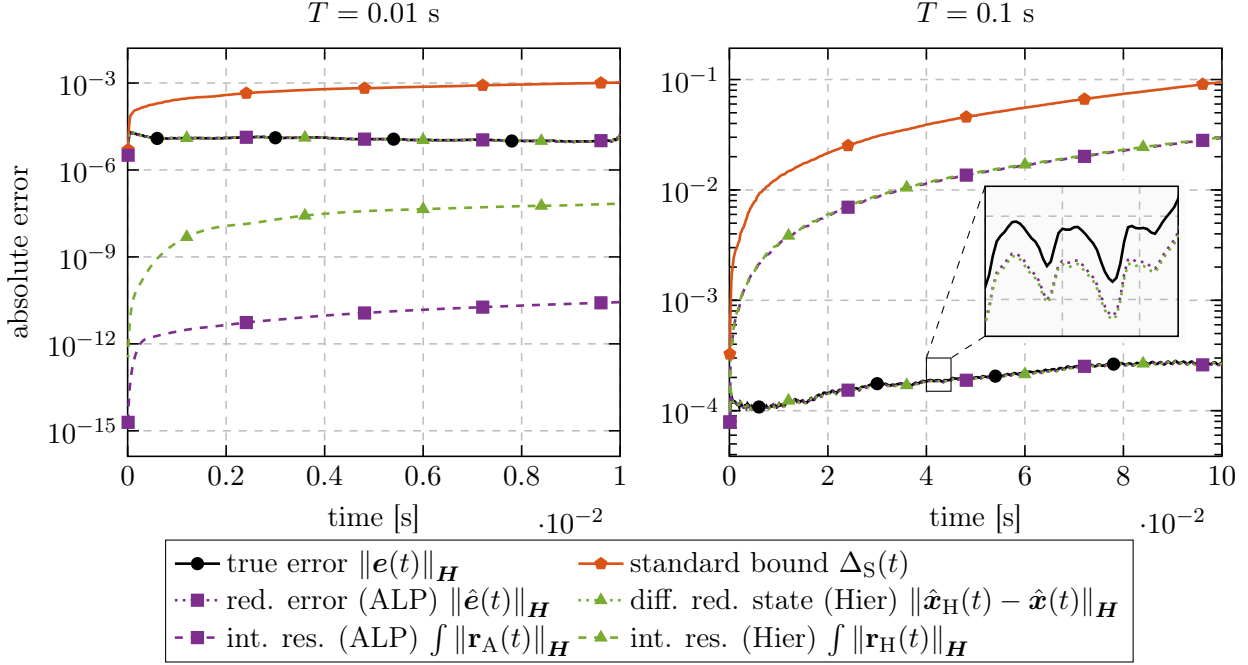}
  	\caption[]
  	{\small Error bounds for a \textit{POD-State} basis generation with a primal basis of $n=120$ and secondary basis size of $n_k = 400$ for $k\in\{\text{A},\text{H}\}$ with a hierarchical basis from additional POD-modes and an ALP basis from error snapshots. A short simulation time $T=0.01~\text{s}$ (left) and a long simulation time $T=0.1~\text{s}$ (right) are shown.} 
  	\label{fig:BoundsSubparts}
  \end{figure}

Looking at the subparts, we find that the residual terms of both the ALP and the hierarchical error lead to larger overestimations of the true error. While the residuals for the short simulation time are negligible because the dynamics can be reproduced very well, the residuals for the longer simulation time enlarge. However, the additional terms, namely the approximated error and the difference of the reduced states, shown in the zoomed-in axis, are very close to the true error. It should be noted that these values are not rigorous bounds but simply error estimators, which can also lead to $\effgen_k<1$, cf.~\cref{fig:BoundsSubparts}. Both, the first summand of the hierarchical and the ALP bounds, would be very suitable as error indicators. Using the approximated error as an error indicator in combination with a greedy procedure to select data samples for constructing the reduced order model (ROM) has also been proposed in~\cite{ChellappaEtAl21}.

\section{Conclusion and outlook}
\label{sec:conclusion_outlook}
In this work, we have adapted and applied existing error bounds, namely the standard error bound, the auxiliary linear problem error bound, and a hierarchical error bound, to pH systems. We thereby exploited the pH system matrix properties to circumvent the computationally demanding calculation of the matrix exponential. Theoretically and numerically, we have proven that in the linear case, the improved error bounds, i.e.~ALP and hierarchical error bounds, are equivalent under a specific choice of basis and initial conditions. Various numerical experiments have been performed for a three-dimensional model of a guitar with fluid-structure interaction. The results have shown effectivities close to one for short simulation times and, compared to the standard bound, improved error bounds for long simulation times that incorporate more complex dynamics. Finally, it was discussed how subcomponents of the improved error bounds could be used as error estimators. Future emphasis will focus on applying error bounds for the ROM construction via greedy procedures, similar to what is done in~\cite{BuchfinkEtAl20} as a PSD-greedy procedure and in~\cite{HainEtAl19} for the hierarchical error bound for inf-sup stable RB problems. Furthermore, attempts can be made to counteract the degradation of the reduction quality for longer simulation times by using time-partitioned bases. \\

%
%
\acknowledgements{
Supported by Deutsche Forschungsgemeinschaft (DFG, German Research Foundation) Project No. 314733389, and under Germany's Excellence Strategy -
EXC 2075 – 390740016. We acknowledge the support by the Stuttgart Center for Simulation Science (SimTech).}


\end{document}